\documentclass[12pt]{amsart}
\usepackage{amsmath}
\usepackage{amsxtra}
\usepackage{amscd}
\usepackage{amsthm}
\usepackage{amsfonts}
\usepackage{amssymb}
\usepackage{eucal}
\usepackage{verbatim}
\usepackage[all]{xy}
\usepackage{color}
\definecolor{deepjunglegreen}{rgb}{0.0, 0.29, 0.29}
\definecolor{darkspringgreen}{rgb}{0.09, 0.45, 0.27}
\definecolor{Red}{rgb}{0.7, 0,0}

\makeatletter
\usepackage{etex,etoolbox,needspace}
\pretocmd\section{\Needspace*{4\baselineskip}}{}{}
\makeatletter

\usepackage[pdftex,bookmarks=false,colorlinks=true,citecolor=darkspringgreen,debug=true,
  naturalnames=true,pdfnewwindow=true]{hyperref}

\newtheorem{thm}{Theorem}[subsection]

\newtheorem{cor}[thm]{Corollary}

\newtheorem{lem}[thm]{Lemma}

\newtheorem{conj}[thm]{Conjecture}

\theoremstyle{definition}

\theoremstyle{remark}
\newtheorem{rem}[thm]{Remark}
\theoremstyle{remark}
\newtheorem{ex}[thm]{Example}

\newcommand{\nc}{\newcommand}
\nc{\renc}{\renewcommand} \nc{\ssec}{\subsection}
\nc{\sssec}{\subsubsection}

\nc{\on}{\operatorname} \nc{\wh}{\widehat}
\nc\ol{\overline} \nc\ul{\underline} \nc\wt{\widetilde}

\newcommand{\red}[1]{{\color{Red}#1}}

\nc{\BA}{{\mathbb{A}}} \nc{\BC}{{\mathbb{C}}} \nc{\BF}{{\mathbb{F}}}
\nc{\BD}{{\mathbb{D}}} \nc{\BG}{{\mathbb{G}}} \nc{\BQ}{{\mathbb{Q}}}
\nc{\BM}{{\mathbb{M}}} \nc{\BN}{{\mathbb{N}}} \nc{\BO}{{\mathbb{O}}}
\nc{\BP}{{\mathbb{P}}} \nc{\BR}{{\mathbb{R}}}
\nc{\BZ}{{\mathbb{Z}}} \nc{\BS}{{\mathbb{S}}} \nc{\BW}{{\mathbb{W}}}

\nc{\CA}{{\mathcal{A}}} \nc{\CB}{{\mathcal{B}}} \nc{\CalC}{{\mathcal{C}}} \nc{\CalD}{{\mathcal{D}}}
\nc{\CE}{{\mathcal{E}}} \nc{\CF}{{\mathcal{F}}}
\nc{\CG}{{\mathcal{G}}} \nc{\CH}{{\mathcal{H}}}
\nc{\CI}{{\mathcal{I}}} \nc{\CK}{{\mathcal{K}}} \nc{\CL}{{\mathcal{L}}}
\nc{\CM}{{\mathcal{M}}} \nc{\CN}{{\mathcal{N}}}
\nc{\CO}{{\mathcal{O}}} \nc{\CP}{{\mathcal{P}}}
\nc{\CQ}{{\mathcal{Q}}} \nc{\CR}{{\mathcal{R}}}
\nc{\CS}{{\mathcal{S}}} \nc{\CT}{{\mathcal{T}}}
\nc{\CU}{{\mathcal{U}}} \nc{\CV}{{\mathcal{V}}}  \nc{\CX}{{\mathcal X}}
\nc{\CY}{{\mathcal Y}} \nc{\CW}{{\mathcal{W}}} \nc{\CZ}{{\mathcal{Z}}}

\nc{\cM}{{\check{\mathcal M}}{}} \nc{\csM}{{\check{\mathcal A}}{}}
\nc{\oM}{{\overset{\circ}{\mathcal M}}{}}
\nc{\obM}{{\overset{\circ}{\mathbf M}}{}}
\nc{\oCA}{{\overset{\circ}{\mathcal A}}{}}
\nc{\obA}{{\overset{\circ}{\mathbf A}}{}}
\nc{\ooM}{{\overset{\circ}{M}}{}}
\nc{\osM}{{\overset{\circ}{\mathsf M}}{}}
\nc{\vM}{{\overset{\bullet}{\mathcal M}}{}}
\nc{\nM}{{\underset{\bullet}{\mathcal M}}{}}
\nc{\oD}{{\overset{\circ}{\mathcal D}}{}}
\nc{\obD}{{\overset{\circ}{\mathbf D}}{}}
\nc{\oA}{{\overset{\circ}{\mathbb A}}{}}
\nc{\op}{{\overset{\bullet}{\mathbf p}}{}}
\nc{\cp}{{\overset{\circ}{\mathbf p}}{}}
\nc{\oU}{{\overset{\bullet}{\mathcal U}}{}}
\nc{\ofZ}{{\overset{\circ}{\mathfrak Z}}{}}

\nc{\fa}{{\mathfrak{a}}} \nc{\fb}{{\mathfrak{b}}}
\nc{\fd}{{\mathfrak{d}}} \nc{\fe}{{\mathfrak{e}}} \nc{\ff}{{\mathfrak{f}}}
\nc{\fg}{{\mathfrak{g}}} \nc{\fgl}{{\mathfrak{gl}}}
\nc{\fh}{{\mathfrak{h}}} \nc{\fri}{{\mathfrak{i}}}
\nc{\fj}{{\mathfrak{j}}} \nc{\fk}{{\mathfrak{k}}} \nc{\fl}{{\mathfrak{l}}}
\nc{\fm}{{\mathfrak{m}}} \nc{\fn}{{\mathfrak{n}}}
\nc{\ft}{{\mathfrak{t}}} \nc{\fu}{{\mathfrak{u}}} \nc{\fv}{{\mathfrak{v}}}
\nc{\fw}{{\mathfrak{w}}} \nc{\fz}{{\mathfrak{z}}}
\nc{\fp}{{\mathfrak{p}}} \nc{\fq}{{\mathfrak{q}}} \nc{\frr}{{\mathfrak{r}}}
\nc{\fs}{{\mathfrak{s}}} \nc{\fsl}{{\mathfrak{sl}}}
\nc{\fso}{{\mathfrak{so}}} \nc{\fsp}{{\mathfrak{sp}}} \nc{\osp}{{\mathfrak{osp}}}
\nc{\hsl}{{\widehat{\mathfrak{sl}}}}
\nc{\hgl}{{\widehat{\mathfrak{gl}}}}
\nc{\hg}{{\widehat{\mathfrak{g}}}}
\nc{\chg}{{\widehat{\mathfrak{g}}}{}^\vee}
\nc{\hn}{{\widehat{\mathfrak{n}}}}
\nc{\chn}{{\widehat{\mathfrak{n}}}{}^\vee}

\nc{\fA}{{\mathfrak{A}}} \nc{\fB}{{\mathfrak{B}}} \nc{\fC}{{\mathfrak{C}}}
\nc{\fD}{{\mathfrak{D}}} \nc{\fE}{{\mathfrak{E}}}
\nc{\fF}{{\mathfrak{F}}} \nc{\fG}{{\mathfrak{G}}} \nc{\fH}{{\mathfrak{H}}}
\nc{\fI}{{\mathfrak{I}}} \nc{\fJ}{{\mathfrak{J}}}
\nc{\fK}{{\mathfrak{K}}} \nc{\fL}{{\mathfrak{L}}}
\nc{\fM}{{\mathfrak{M}}} \nc{\fN}{{\mathfrak{N}}}
\nc{\frP}{{\mathfrak{P}}} \nc{\fQ}{{\mathfrak{Q}}}
\nc{\fS}{{\mathfrak{S}}} \nc{\fT}{{\mathfrak{T}}} \nc{\fU}{{\mathfrak{U}}}
\nc{\fV}{{\mathfrak{V}}} \nc{\fW}{{\mathfrak{W}}}
\nc{\fX}{{\mathfrak{X}}} \nc{\fY}{{\mathfrak{Y}}}
\nc{\fZ}{{\mathfrak{Z}}}

\nc{\ba}{{\mathbf{a}}}
\nc{\bb}{{\mathbf{b}}} \nc{\bc}{{\mathbf{c}}} \nc{\be}{{\mathbf{e}}}
\nc{\bg}{{\mathbf{g}}} \nc{\bj}{{\mathbf{j}}} \nc{\bm}{{\mathbf{m}}}
\nc{\bn}{{\mathbf{n}}} \nc{\bp}{{\mathbf{p}}}
\nc{\bq}{{\mathbf{q}}} \nc{\br}{{\mathbf{r}}} \nc{\bt}{{\mathbf{t}}}
\nc{\bfu}{{\mathbf{u}}} \nc{\bv}{{\mathbf{v}}}
\nc{\bx}{{\mathbf{x}}} \nc{\by}{{\mathbf{y}}} \nc{\bz}{{\mathbf{z}}}
\nc{\bw}{{\mathbf{w}}} \nc{\bA}{{\mathbf{A}}}
\nc{\bB}{{\mathbf{B}}} \nc{\bC}{{\mathbf{C}}}
\nc{\bD}{{\mathbf{D}}} \nc{\bF}{{\mathbf{F}}} \nc{\bG}{{\mathbf{G}}}
\nc{\bH}{{\mathbf{H}}} \nc{\bI}{{\mathbf{I}}} \nc{\bJ}{{\mathbf{J}}}
\nc{\bK}{{\mathbf{K}}} \nc{\bL}{{\mathbf{L}}} \nc{\bM}{{\mathbf{M}}}
\nc{\bN}{{\mathbf{N}}}
\nc{\bO}{{\mathbf{O}}} \nc{\bS}{{\mathbf{S}}} \nc{\bT}{{\mathbf{T}}}
\nc{\bU}{{\mathbf{U}}} \nc{\bV}{{\mathbf{V}}} \nc{\bW}{{\mathbf{W}}}
\nc{\bX}{{\mathbf{X}}}
\nc{\bY}{{\mathbf{Y}}} \nc{\bP}{{\mathbf{P}}}
\nc{\bZ}{{\mathbf{Z}}} \nc{\bh}{{\mathbf{h}}}

\nc{\sA}{{\mathsf{A}}} \nc{\sB}{{\mathsf{B}}}
\nc{\sC}{{\mathsf{C}}} \nc{\sD}{{\mathsf{D}}}
\nc{\sE}{{\mathsf{E}}} \nc{\sF}{{\mathsf{F}}} \nc{\sG}{{\mathsf{G}}} \nc{\sH}{{\mathsf{H}}}
\nc{\sI}{{\mathsf{I}}} \nc{\sK}{{\mathsf{K}}} \nc{\sL}{{\mathsf{L}}}
\nc{\sfm}{{\mathsf{m}}} \nc{\sM}{{\mathsf{M}}} \nc{\sN}{{\mathsf{N}}}
\nc{\sO}{{\mathsf{O}}} \nc{\sQ}{{\mathsf{Q}}} \nc{\sP}{{\mathsf{P}}}
\nc{\sT}{{\mathsf{T}}} \nc{\sZ}{{\mathsf{Z}}}
\nc{\sV}{{\mathsf{V}}} \nc{\sW}{{\mathsf{W}}}
\nc{\sfp}{{\mathsf{p}}} \nc{\sq}{{\mathsf{q}}} \nc{\sr}{{\mathsf{r}}}
\nc{\sfs}{{\mathsf{s}}} \nc{\st}{{\mathsf{t}}} \nc{\sfb}{{\mathsf{b}}}
\nc{\sfc}{{\mathsf{c}}} \nc{\sd}{{\mathsf{d}}}
\nc{\sz}{{\mathsf{z}}}

\nc{\tA}{{\widetilde{\mathbf{A}}}}
\nc{\tB}{{\widetilde{\mathcal{B}}}}
\nc{\tg}{{\widetilde{\mathfrak{g}}}} \nc{\tG}{{\widetilde{G}}}
\nc{\TM}{{\widetilde{\mathbb{M}}}{}}
\nc{\tO}{{\widetilde{\mathsf{O}}}{}}
\nc{\tU}{{\widetilde{\mathfrak{U}}}{}} \nc{\TZ}{{\tilde{Z}}}
\nc{\tx}{{\tilde{x}}} \nc{\tbv}{{\tilde{\bv}}}
\nc{\tfP}{{\widetilde{\mathfrak{P}}}{}} \nc{\tz}{{\tilde{\zeta}}}
\nc{\tmu}{{\tilde{\mu}}}

\nc{\urho}{\underline{\rho}} \nc{\uB}{\underline{B}}
\nc{\uC}{{\underline{\mathbb{C}}}} \nc{\ui}{\underline{i}}
\nc{\uj}{\underline{j}} \nc{\ofP}{{\overline{\mathfrak{P}}}}
\nc{\oB}{{\overline{\mathcal{B}}}}
\nc{\og}{{\overline{\mathfrak{g}}}} \nc{\oI}{{\overline{I}}}

\nc{\eps}{\varepsilon} \nc{\hrho}{{\hat{\rho}}} \nc{\balpha}{{\boldsymbol{\alpha}}}
\nc{\blambda}{{\boldsymbol{\lambda}}} \nc{\bmu}{{\boldsymbol{\mu}}} \nc{\bnu}{{\boldsymbol{\nu}}}
\nc{\btheta}{{\boldsymbol{\theta}}} \nc{\bzeta}{{\boldsymbol{\zeta}}} \nc{\bta}{{\boldsymbol{\eta}}}
\nc{\bbeta}{{\boldsymbol{\beta}}} \nc{\bkappa}{{\boldsymbol{\kappa}}} \nc{\bomega}{{\boldsymbol{\omega}}}

\nc{\one}{{\mathbf{1}}} \nc{\two}{{\mathbf{t}}}

\nc\DMO\DeclareMathOperator
\DMO\Sym{Sym}
\nc{\Tot}{{\mathop{\operatorname{\rm Tot}}}}
\nc{\Spec}{\mathop{\operatorname{\rm Spec}}}
\nc{\Ker}{{\mathop{\operatorname{\rm Ker}}}}
\nc{\Isom}{{\mathop{\operatorname{\rm Isom}}}}
\nc{\Hilb}{{\mathop{\operatorname{\rm Hilb}}}}
\nc{\deeq}{{\mathop{\operatorname{\rm deeq}}}}
\nc{\End}{{\mathop{\operatorname{\rm End}}}}
\nc{\Ext}{{\mathop{\operatorname{\rm Ext}}}}
\nc{\Hom}{{\mathop{\operatorname{\rm Hom}}}}
\nc{\CHom}{{\mathop{\operatorname{{\mathcal{H}}\it om}}}}
\nc{\GL}{{\mathop{\operatorname{\rm GL}}}}
\nc{\PGL}{{\mathop{\operatorname{\rm PGL}}}}
\nc{\SL}{{\mathop{\operatorname{\rm SL}}}}
\nc{\SO}{{\mathop{\operatorname{\rm SO}}}}
\nc{\Sp}{{\mathop{\operatorname{\rm Sp}}}}
\nc{\OSp}{{\mathop{\operatorname{\rm SOSp}}}}
\nc{\gr}{{\mathop{\operatorname{\rm gr}}}}
\nc{\Id}{{\mathop{\operatorname{\rm Id}}}}
\nc{\perf}{{\mathop{\operatorname{\rm perf}}}}
\nc{\defi}{{\mathop{\operatorname{\rm def}}}}
\nc{\length}{{\mathop{\operatorname{\rm length}}}}
\nc{\supp}{{\mathop{\operatorname{\rm supp}}}}
\nc{\HC}{{\mathcal H}{\mathcal C}}
\nc{\Vect}{{\mathop{\operatorname{\rm Vect}}}}

\nc{\pr}{{\operatorname{pr}}}
\nc{\Cliff}{{\mathsf{Cliff}}}
\nc{\loc}{{\operatorname{loc}}} \nc{\lc}{{\operatorname{lc}}}
\nc{\Fl}{{\mathbf{Fl}}} \nc{\Ffl}{{\mathcal{F}\ell}}
\nc{\Fib}{{\mathsf{Fib}}}
\nc{\Coh}{{\mathsf{Coh}}} \nc{\FCoh}{{\mathsf{FCoh}}}
\nc{\Perf}{{\mathsf{Perf}}}
\nc{\wtimes}{\mathbin{\widetilde\times}}

\nc{\reg}{{\text{\rm reg}}} \nc{\ren}{{\text{\rm ren}}}
\nc{\self}{{\text{\rm self}}}
\nc{\gvee}{{\mathfrak g}^{\!\scriptscriptstyle\vee}}
\nc{\tvee}{{\mathfrak t}^{\!\scriptscriptstyle\vee}}
\nc{\nvee}{{\mathfrak n}^{\!\scriptscriptstyle\vee}}
\nc{\bvee}{{\mathfrak b}^{\!\scriptscriptstyle\vee}}
       \nc{\rhovee}{\rho^{\!\scriptscriptstyle\vee}}

\nc{\cplus}{{\mathbf{C}_+}} \nc{\cminus}{{\mathbf{C}_-}}
\nc{\cthree}{{\mathbf{C}_*}} \nc{\Qbar}{{\bar{Q}}}

\newcommand\iso{\mathbin{\vphantom{j^{X^2}}\smash{\overset{\sim}{\vphantom{\rule{0pt}{0.20em}}\smash{\longrightarrow}}}}}
\nc{\Gtimes}{\vphantom{j^{X^2}}\smash{\overset{G}{\vphantom{\rule{0pt}{0.30em}}\smash{\times}}}}
\nc{\sGtimes}{\vphantom{j^{X^2}}\smash{\overset{\mathsf G}{\vphantom{\rule{0pt}{0.30em}}\smash{\times}}}}

\nc{\bOmega}{{\overline{\Omega}}}

\nc{\seq}[1]{\stackrel{#1}{\sim}}

\nc{\aff}{{\operatorname{aff}}}
\nc{\fin}{{\operatorname{fin}}}
\nc{\mir}{{\operatorname{mir}}}
\nc{\triv}{{\operatorname{triv}}}
\nc{\ext}{{\operatorname{ext}}}
\nc{\righ}{{\operatorname{right}}}
\nc{\lef}{{\operatorname{left}}}
\nc{\forg}{{\operatorname{forg}}}
\nc{\fid}{{\operatorname{fd}}}
\nc{\odd}{{\operatorname{odd}}}
\nc{\even}{{\operatorname{even}}}
\nc{\modu}{{\operatorname{-mod}}}
\nc{\Gr}{{\operatorname{Gr}}}
\nc{\FT}{{\operatorname{FT}}}
\nc{\Mat}{{\operatorname{Mat}}}
\nc{\MSt}{{\operatorname{MSt}}}
\nc{\sph}{{\operatorname{sph}}}
\nc{\GR}{{\mathbf{Gr}}}
\nc{\Perv}{{\operatorname{Perv}}}
\nc{\Rep}{{\operatorname{Rep}}}
\nc{\Ind}{{\operatorname{Ind}}}
\nc{\IC}{{\operatorname{IC}}}
\nc{\Bun}{{\operatorname{Bun}}}
\nc{\Proj}{{\operatorname{Proj}}}
\nc{\Stab}{{\operatorname{Stab}}}
\nc{\pt}{{\operatorname{pt}}}
\nc{\bfmu}{{\boldsymbol{\lambda}}}
\nc{\bfomega}{{\boldsymbol{\omega}}}
\nc{\calM}{\mathcal M}
\nc{\calA}{\mathcal A}
\nc{\calO}{\mathcal O}
\nc{\CC}{\mathcal C}
\nc{\calN}{\mathcal N}
\nc{\grg}{\mathfrak g}
\nc{\dslash}{/\!\!/}
\nc{\tslash}{/\!\!/\!\!/}
\nc\grt{\mathfrak t}
\nc\bfM{\mathbf M}
\nc\bfN{\mathbf N}
\nc\Sig{\Sigma}
\nc\ZZ{\mathbb{Z}}
\nc\calC{\mathcal C}
\nc\calF{\mathcal F}
\nc\calX{\mathcal X}
\nc\calY{\mathcal Y}
\nc\QCoh{\operatorname{QCoh}}
\nc\IndCoh{\operatorname{IndCoh}}
\nc\Maps{\operatorname{Maps}}
\nc\Dmod{D-\operatorname{mod}}
\newcommand\Hecke{\operatorname{Hecke}}
\nc{\calD}{\mathcal D}
\nc\bfO{\mathbf O}

\nc\GG{\mathbb G}
\nc\calK{\mathcal K}
\nc{\calG}{\mathcal G}
\nc\RHom{\operatorname{RHom}}
\nc\Res{\operatorname{Res}}
\nc\Av{\operatorname{Av}}
\nc{\RH}{{\operatorname{RH}}}
\nc{\RT}{{\operatorname{RT}}}
\nc{\DR}{{\operatorname{DR}}}
\nc{\Ome}{\Omega}
\nc\Lam{\Lambda}

\nc\grs{\mathfrak s}
\nc{\tilX}{\widetilde X}
\nc\calB{\mathcal B}
\nc\calS{\mathcal S}
\nc\calT{\mathcal T}
\nc\calZ{\mathcal Z}
\nc\LS{\operatorname{LocSys}}
\nc\bfL{\on{\mathbf L}}

\newcommand*\circled[1]
{*{#1}}
\makeatletter
\newcommand{\raisemath}[1]{\mathpalette{\raisem@th{#1}}}
\newcommand{\raisem@th}[3]{\raisebox{#1}{$#2#3$}}
\nc{\binlim}[2][]{\def\@tempa{#1}\@ifnextchar^{\@binlim{#2}}{\@binlim{#2}^{}}}
\def\@binlim#1^#2{\mathbin{\@ifempty{#2}{\mathop{#1}}{\mathop{#1}\@xp\displaylimits\@tempa^{#2}}}}

\makeatother

\nc\cX{{\mathcal X}}

\nc\Gm{{\mathbb G_m}}
\nc{\isoto}
    {\stackrel{\displaystyle\raisemath{-.2ex}{\sim}}\to}
\nc\cD\CalD
\nc\setm\setminus
\nc\cE\CE
\renc\Hecke{\mathit{\CH\kern-.2ex ecke}}
\nc\bsl\backslash
\nc\Fq{\mathbb F_q}
\nc\x\times
\nc\bGO{{\bG_\bO}}
\nc\bla\blambda
\nc\blt\bullet
\nc\opp{{\on{op}}}
\nc\srel\stackrel
\nc\oast\circledast
\nc\tbx{\binlim{\widetilde\boxtimes{}}}
\nc\into\hookrightarrow
\nc\otto\leftrightarrow
\renc\setminus\smallsetminus
\nc\phitau{\varphi\tau}

\usepackage{tikz}
\usepackage{adjustbox} 

\newcommand{\dashedslash}{%
  \adjustbox{valign=m}{%
    \begin{tikzpicture}[scale=0.30, baseline]
      \draw[dashed, dash pattern=on 1pt off 1pt, line width=0.3pt] (0,0) -- (0.5,1);
    \end{tikzpicture}%
  }%
}

\newcommand{\dashedbackslash}{%
  \adjustbox{valign=m}{%
    \begin{tikzpicture}[scale=0.30, baseline]
      \draw[dashed, dash pattern=on 1pt off 1pt, line width=0.3pt] (0.5,0) -- (0,1);
    \end{tikzpicture}%
  }%
}

\newenvironment{i-ii-iii}{%
\begin{enumerate}
}%
{\end{enumerate}}
\nc\ceil[1]{\lceil#1\rceil}  \nc\floor[1]{\lfloor#1\rfloor}
\nc\Lie{\on{Lie}}
\nc\sS{{\mathsf S}}
\nc\vvv{\ensuremath{\red\surd}}
\nc\la\lambda
\def\arxiv#1{\href{http://arxiv.org/abs/#1}{\tt arXiv:#1}} 
\nc\kap{\kappa}
\nc\gra{\mathfrak a}
\nc\gl{\mathfrak{gl}}
\nc\sTr{\operatorname{sTr}}
\nc\hatG{\widehat{G}}
\nc\calL{\mathcal L}
\nc\Whit{\operatorname{Whit}}
\nc\KL{\operatorname{KL}}

\makeatletter
\renewcommand{\subsection}{\@startsection{subsection}{2}{0pt}{-3ex
plus -1ex minus -0.2ex}{-2mm plus -0pt minus
    -2pt}{\normalfont\bfseries}} \makeatother

\nc{\svee}{{\!\scriptscriptstyle\vee}}

\numberwithin{equation}{subsection}
\allowdisplaybreaks
\nc\mto{\mapsto }
\nc\en{\enspace }
\nc\DD{\mathbb{D}}
\nc\Ztwo{\mathbb Z/2 }
\newcommand\hatR{\widehat{R}}
\nc\grb{\mathfrak b}
\nc\tilD{\widetilde D}
\renewcommand\mod{{\operatorname{-mod}}}
\newcommand\mon{{\operatorname{mon}}}
\nc\alp{\alpha}

\begin{document}

\author[A.Braverman]{Alexander Braverman}
\address{Department of Mathematics, University of Toronto and Perimeter Institute
of Theoretical Physics, Waterloo, Ontario, Canada, N2L 2Y5}
\email{braval@math.toronto.edu}

\author[M.Finkelberg]{Michael Finkelberg}
\address{Einstein Institute of Mathematics, The Hebrew University of Jerusalem,
  Edmond J. Safra Campus, Giv’at Ram, Jerusalem, 91904, Israel;
\newline  National Research University Higher School of Economics}
\email{fnklberg@gmail.com}


\author[R.Travkin]{Roman Travkin}
\address{Skolkovo Institute of Science and Technology, Moscow, Russia}
\email{roman.travkin2012@gmail.com}

\title
{Relative Langlands duality and Koszul duality}




\begin{abstract}
  Consider a pair of $S$-dual hyperspherical varieties $G\circlearrowright X$ and $G^\vee\circlearrowright X^\vee$ equipped with
  equivariant quantizations $Q(X)$, $Q(X^\vee)$. Assume that the local conjecture of Ben-Zvi, Sakellaridis and Venkatesh holds
  for this pair, and also that $X\simeq T^*_\psi(Y)$ is polarized, so that $Q(X)=D_\psi(Y)$. Let $B\subset G$
  (resp.\ $B^\vee\subset G^\vee$) be Borel subgroups. Then
  using a variant of the $S^1$-equivariant localization of~\cite{bzn}, 
  we deduce an equivalence between the $\BZ/2$-graded $B$-equivariant
  category $(D_\psi(Y)\mod^B)^{\BZ/2}$ and the $\BZ/2$-graded unipotent $B^\vee$-monodromic category
  $(Q(X^\vee)\mod^{B^\vee,\operatorname{mon}})^{\BZ/2}$.  
\end{abstract}

\maketitle

\tableofcontents

\section{Introduction}
\subsection{Relative Langlands duality and equivalence of spherical categories}
Let $G$ be a connected reductive group over $\BC$ and let  $\fY$ be a smooth affine spherical $G$-variety -- i.e.\ a smooth affine $G$-variety with an open dense $B$-orbit, where $B$ is a Borel subgroup of $G$. We shall make an additional assumption: the stabilizer of a general point of $\fY$ in $B$ is connected.

\begin{rem}\label{hyper-spherical}
In fact the setup of \cite{bzsv} is more general. In {\em loc.cit.} the authors start with a {\em hyper-spherical}  affine Poisson $G$-variety $\fX$ such that the stabilizer of a general point in $G$ is connected. If one starts with a spherical variety $\fY$ as above then one can attach to it a hyper-spherical variety $\fX=T^*\fY$. Moreover, it is shown in \cite{bzsv} that the stabilizer of a general point of $\fX$ in $G$ is connected if and only if the stabilizer of a general point of $\fY$ in $B$ is connected. This language makes what follows much more symmetric, but the main results of this paper will be concerned with the cotangent case, i.e.\ the case when $\fX=T^*\fY$ as above, so we shall restrict ourselves to this setup from the very beginning.
\end{rem}

One of the main conjectures of \cite{bzsv} is that there should exist a dual hyper-spherical affine Poisson $G^{\vee}$-variety $\fX^{\vee}$ endowed with the following structures:

1) A $\mathbb C^\times$-action so that the Poisson structure has degree 2.

2) A graded quantization $Q_\hbar(\fX^{\vee})$ which we are going to consider as a dg algebra with 0 differential.

\noindent
Moreover, it is expected that the following holds. First for an algebraic group $A$ acting on a scheme $Z$ we denote by $D(Z)^A$ the
derived category of $A$-equivariant $D$-modules on $Z$ and by $\tilD(Z)^A$ the corresponding renormalized  equivariant derived
category -- i.e.\ the cocompletion of the category consisting of objects of the usual equivariant derived category which have
coherent cohomology  in finitely many degrees (i.e.\ the category of locally compact objects in $D(Z/A)$);
the difference between the renormalized equivariant derived category and the usual
equivariant derived category is discussed e.g.\ in \cite{dg}). Typically we are going to deal with scheme $Z$ where $A$ has finitely
many orbits (or and ind-scheme which is an a colimit of such schemes), so in that case we can work with the corresponding derived
category of equivariant constructible sheaves.
\begin{conj}\label{bzsv}
  The category $\tilD(\fY(\CK))^{G(\CO)\rtimes \mathbb C^\times}$ is equivalent to $(Q_{\hbar}(\fX^{\vee})\mod)^{G^{\vee}}$ (the derived category
  of dg-modules over the corresponding dg-algebra endowed with a compatible action of $G^{\vee}$). Here $\mathbb C^\times$ on the left is
  the ``loop rotation", and the corresponding equivariant parameter corresponds to $\hbar$ on the right.
\end{conj}
In fact, the main Conjecture of \cite{bzsv} is more general: instead of starting with a spherical variety $\fY$ one can start with a
hyperspherical $G$-variety $\fX$ (in the above situation one can take $\fX=T^*\fY$). This makes the story more symmetric, but also more
complicated. We shall refer to the situation when $\fX=T^*\fY$ as the cotangent case. In this paper we shall be mostly concerned with
that situation or with the so called twisted cotangent case -- cf.~\S\ref{twisted}.

\subsection{Koszul duality conjecture}
In addition to Conjecture \ref{bzsv} the following conjecture was formulated in \cite{fgt}.
First let $R$ be a (non-commutative) ring endowed with a strong action of an algebraic group $B$; by definition this means
that $B$ acts on $R$ and there is a homomorphism of Lie algebras $\iota\colon  \grb\to R$ so that the derivative of the
$B$-action is equal to the adjoint action of $\iota$. Then  one can talk about the category of monodromic $R$-modules with
unipotent monodromy.
By definition this is the category of modules $M$ over $R$ endowed with a compatible action of $B$ so that difference between the
derivative of this action and the action of $\grb$ on $M$ coming from $\iota$ is equal to the derivative of some fixed character
$\lambda\colon B\to {\mathbb G}_m$ (it is easy to see that this definition does not depend on the choice of $\lambda$).
\begin{conj}\label{koszul}
  Let $B$ denote a Borel subroup of $G$, and let $B^{\vee}$ be a Borel subgroup of $G^{\vee}$. Let $(Q_{\hbar=1}(\fX^{\vee})\mod)^{B^{\vee}_\mon}$
  denote the derived category of $B^{\vee}$-monodromic modules over
  the (non-graded) algebra $Q_{\hbar=1}(\fX^{\vee})$ with unipotent monodromy. Then the categories $D(\fY)^B$ and
  $(Q_{\hbar=1}(\fX^{\vee})\mod)^{B^{\vee}_\mon}$ are Koszul dual. Similarly, the category $D(\fY)^{B_\mon}$ is Koszul dual to
  $(Q_{\hbar=1}(\fX^{\vee})\mod)^{B^{\vee}}$.
\end{conj}

Note that Koszul duality in particular means that the two categories above should be endowed with some additional gradings (and the duality is essentially an equivalence of bigraded categories which mixes the gradings in some particular way).
To avoid this issue one can pass to the world of $\BZ/2$-graded categories. For a $\ZZ$-graded category $D$ we shall denote by $D^{\BZ/2}$ the corresponding $\BZ/2$-graded category. Then we get the following version of Conjecture \ref{koszul}:
\begin{conj}\label{conjZ2}
The $B$-equivariant category $(D(\fY)^B)^{\BZ/2}$ is equivalent to the $B^\vee$-unipotent monodromic category
$(Q_{\hbar=1}(\fX^{\vee})\mod^{B^{\vee}_\mon})^{\BZ/2}$.\footnote{Completions of $\BZ/2$-graded categories are briefly discussed
in~\S\ref{completion}.} Similarly, the category $(D(\fY)^{B_\mon})^{\BZ/2}$ is equivalent to
$(Q_{\hbar=1}(\fX^{\vee})\mod^{B^{\vee}})^{\BZ/2}$.
\end{conj}

\subsection{What is done in this paper}The purpose of this paper is to develop a general framework for proving Conjecture \ref{conjZ2}.
This general formalism will show that Conjecture \ref{conjZ2} follows from Conjecture \ref{bzsv}. 


\noindent
Let us explain this in more detail. What we are actually going to prove is this:
\begin{thm}\label{main}
\begin{enumerate}
\item
Assume the validity of Conjecture \ref{bzsv} for given $\fY$ and $\fX^{\vee}$. Then
the category $(D(\fY)^B)^{\BZ/2}$ is equivalent to $(Q_{\hbar=1}(\fX^{\vee})\mod^{B^{\vee}_\mon})^{\BZ/2}$.
\item
  Similar results hold when $(D(\fY)^B)^{\BZ/2}$  is replaced by $(D(\fY)^{B_\mon})^{\BZ/2}$, and
  $(Q_{\hbar=1}(\fX^{\vee})\mod^{B^{\vee}_\mon})^{\BZ/2}$ is replaced by $(Q_{\hbar=1}(\fX^{\vee})\mod^{B^{\vee}})^{\BZ/2}$.
\end{enumerate}
\end{thm}

\subsection{The Hecke action}\label{hecke} One can strengthen Conjecture \ref{koszul} in the following way. The category $D_B(\fY)$ carries an action of the monoidal category $D(B\backslash G/B)$ (i.e.\ the category of $B\times B$-equivariant $D$-modules on $G$. We shall usually refer to this category as the {\em Hecke category} of $G$). Also, let us denote by $D(B^{\vee}\dashedbackslash \, G^{\vee}\dashedslash B^{\vee})$ the category of $B^{\vee}\times B^{\vee}$ monodromic $D$-modules on $G^{\vee}$ with unipotent monodromy; we shall refer to it as the {\em monodromic Hecke category} for $G^{\vee}$. This is also a monoidal category, and this monoidal category acts on $(Q_{\hbar=1}(\fX^{\vee})\mod)^{B^{\vee}_\mon}$.

On the other hand, R.~Bezrukavnikov and Z.~Yun (cf.~\cite{by}) established a monoidal Koszul duality between $D(B\backslash G/B)$ and $D(B^{\vee}\dashedbackslash \, G^{\vee}\dashedslash B^{\vee})$. The stronger version of Conjecture \ref{koszul} says that
the Koszul duality of Conjecture \ref{koszul} is compatible with the action of the Hecke category on the LHS and the monodromic Hecke category on the RHS via the Bezrukavnikov-Yun duality.
In particular, the same should be true on the level of $\BZ/2$-graded categories if one replaces ``Koszul duality" with ``equivalence".

On the other hand, the duality of Bezrukavnikov and Yun can by itself be considered essentially as a special case of Conjecture \ref{koszul}.
Namely, let $\fY=G$ considered as a spherical $G\times G$-variety. In this case $\fX^{\vee}=T^*\fY^{\vee}$ where $\fY^{\vee}=G^{\vee}$ considered as a spherical $G^{\vee}\times G^{\vee}$-variety. Then the LHS of the duality in~Conjecture~\ref{koszul} is precisely $D(B\backslash G/B)$ while the RHS is $D(B^{\vee}\dashedbackslash \, G^{\vee}\dashedslash B^{\vee})$.
Note that~Conjecture~\ref{bzsv} is known in this case -- it is precisely the statement of the derived geometric Satake equivalence.

To formulate the precise statement, we will need the Radon transform (the long intertwining functor)
$I_{w_0}^{-1}\colon D(B^{\vee}\dashedbackslash \, G^{\vee}\dashedslash B^{\vee})\iso D(B^{\vee}\dashedbackslash \, G^{\vee}\dashedslash B^{\vee})$
(see~\S\ref{intertwining} for the precise definiton).

\begin{thm}\label{byun}
  \textup{(1)} The equivalence between $\BZ/2$-graded versions of the Hecke category for $G$ and the monodromic Hecke category for $G^{\vee}$ coming
  from the Bezrukavnikov-Yun Koszul duality is equal to the equivalence
  coming from~Theorem~\ref{main} composed with the functor $I_{w_0}^{-1}$. In particular, the above composition is monoidal.

  \textup{(2)} Recall that both categories in~\textup{(1)} above are linear over $\Sym(\ft^*)\otimes\Sym(\ft^*)$. The Bezrukavnikov-Yun
  equivalence is $\Sym(\ft^*)\otimes\Sym(\ft^*)$-linear, and the composition of the equivalence coming from~Theorem~\ref{main} with
  $I_{w_0}^{-1}$ becomes $\Sym(\ft^*)\otimes\Sym(\ft^*)$-linear after conjugation with the action of $w_0$ in the second variable.
\end{thm}

The following is an easy corollary of the proof of Theorem \ref{byun}:
\begin{cor}\label{hecke-mod}
The equivalence of Theorem \ref{main}(1) is compatible with the action of $D^{\BZ/2}(B\backslash G/B)$ on the LHS and the action of  $D^{\BZ/2}(B^{\vee}\dashedbackslash \, G^{\vee}\dashedslash B^{\vee})$ on the RHS via the $\BZ/2$-version of the Bezrukavnikov-Yun equivalence.
\end{cor}

\subsection{Twisted cotangent case}\label{twisted}
In the above discussion we assumed that $\fX=T^*\fY$ for some spherical $G$-variety. In fact, all the arguments go through in the
following ``twisted" case: assume that we are given a $G$-equivariant $\GG_a$ torsor $\psi$ on $\fY$, and assume that $\fX$ is the
Hamiltonian reduction of $T^*\psi$ by $\GG_a$ at level~$1$ (we shall write $\fX=T_\psi^*\fY$); we shall refer to such an $\fX$ as the
twisted (by $\psi$) cotangent bundle. In this case the category $D(\fY(\CK))^{G(\CO)}$ is replaced by
the corresponding category of twisted (by the exponential $D$-module on the additive group) $G(\CO)$-equivariant $D$-modules on $\fY$;
in other words these are $(\GG_a,\text{exp})$-equivariant $D$-modules on $\psi$. Likewise the category $D(\fY)^B$ is replaced by the category $D(\fY)^{B,\psi}$ of twisted $B$-equivariant $D$-modules on $\fY$. The proofs extend verbatim to this case.

\subsection{An isomorphism of $W$-modules}
Let $W$ denote the Weyl group of $G$. Since it is also the Weyl group of $G^{\vee}$, we have natural identifications
\begin{equation}\label{K}
K_0(D(B\backslash G/B))\simeq \ZZ[W]\simeq K_0(D(B^{\vee}\dashedbackslash \, G^{\vee}\dashedslash B^{\vee})).
\end{equation}
In the above equation one can also pass to the $\BZ/2$-graded versions of all the categories, since this does not change the Grothendieck group.
The Bezrukavnikov-Yun equivalence induces thus an involutive automorphism $\sigma$ of $\ZZ[W]$ sending every $w$ to $(-1)^{\ell(w)}w$.

On the other hand, we set $\fX=T^*_\psi\fY$, and consider a symplectic variety $\fX\times T^*\CB$
equipped with a Hamiltonian $G$-action. (Here $\CB$ stands for the flag variety $\CB=G/B$.)
We denote by $\tilde\bmu\colon \fX\times T^*\CB\to\fg^*$ the moment map. Then the zero moment level
$\widetilde{\Lambda}_\fX:=\tilde\bmu{}^{-1}(0)$ is the union of $\psi$-twisted conormal bundles to the
{\em relevant} $G$-orbits in $\fY\times\CB$. In particular, $\widetilde{\Lambda}_\fX\subset\fX\times T^*\CB$ is
a Lagrangian subvariety. Recall that according to~\cite[Theorem 3.4.1]{cg}, the group algebra
$\BC[W]$ is naturally isomorphic to the convolution algebra $H(Z_\fg)$ spanned by the top degree part
of the Borel-Moore homology of the Steinberg variety of triples $Z_\fg=T^*\CB\times_\fg T^*\CB$.
This algebra acts naturally on $H(\widetilde{\Lambda}_\fX)$: the top degree part of the Borel-Moore
homology of $\widetilde{\Lambda}_\fX$. Similarly, on the dual side, $\BC[W]\cong H(Z_{\fg^\vee})$
acts on $H(\widetilde{\Lambda}_{\fX^\vee})$. In~\cite{fgt} we formulated a conjecture relating the
$H(Z_\fg)\cong\BC[W]\cong H(Z_{\fg^\vee})$-modules $H(\widetilde{\Lambda}_\fX)$ and
$H(\widetilde{\Lambda}_{\fX^\vee})$. This conjecture is false as stated; we are grateful to Guy Shtotland
who drew our attention to this fact, see~\cite[Remark 1.5]{s}. The corrected version reads as follows:

\begin{conj}\label{lagrangian}
  The $\BC[W]$-modules $H(\widetilde{\Lambda}_\fX)^*$ and
  $H(\widetilde{\Lambda}_{\fX^\vee})\otimes\mathrm{sign}$ are isomorphic.\footnote{Note that any finite dimensional $\BC[W]$-module is
  non-canonically isomorphic to its dual.}
\end{conj}

In \S\ref{Lagrange} we explain how to deduce Conjecture \ref{lagrangian} from the previous results in a number of cases by for example
identifying $H(\widetilde{\Lambda})^*$ with $K(D(\fY)^{B,\psi}_{\on{comp}})$ (the Grothendieck group of the category of compact objects),
and $H(\widetilde{\Lambda}^{\vee})$ with $K(D(\fY^{\vee})^{B^{\vee}_\mon,\psi^{\vee}}_{\on{comp}})$ when it is possible
(in fact in \S\ref{Lagrange} something a bit more general is discussed).


\subsection{General monodromy} One can ask what happens if in the RHS of Theorem~\ref{main} we consider more general
(not necessarilly unipotent) monodromy. It turns out that the same methods give rise to the following.
We fix a Cartan subgroup $\sT\subset B$. Let us fix a regular (but not necessarilly integral) coweight $\lambda$ of $G$.
Let $\sT_{\lambda}$ denote the subgroup of $\sT$ generated by $\exp(2\pi i\lambda)$ and let $B_{\lambda}$ be its centralizer
in $B$.  Let $(Q_{\hbar=1}(\fX^{\vee})\mod^{B^{\vee}_{\mon,\lambda}})^{\BZ/2}$ denote the $\BZ/2$-graded version of the category of
$\lambda$-monodromic modules over $Q_{\hbar=1}(\fX^{\vee})$.  Then we have

\begin{thm}\label{main'}
  Under the assumptions of Theorem \ref{main} the categories $(D(\fY^{\sT_{\lambda}})^{B_{\lambda}})^{\BZ/2}$ and
  $(Q_{\hbar=1}(\fX^{\vee})\mod^{B^{\vee}_{\mon,\lambda}})^{\BZ/2}$ are equivalent.
\end{thm}

Of course, we also have

\begin{conj}\label{fgt'}
Under the above assumptions the categories $D(\fY^{\sT_{\lambda}})^{B_{\lambda}}$ and $(Q_{\hbar=1}(\fX^{\vee})\mod)^{B^{\vee}_{\mon,\lambda}}$ are Koszul dual.
\end{conj}


\begin{ex} Let us as before consider the diagonal case, i.e.\ $\fY=G$ acted on by the group $G\times G$ and assume that we start with a regular coweight of $G\times G$ of the form
$(\lambda,\lambda)$ for some regular coweight $\lambda$ of $G$. Let $G_{\lambda}$ denote the centralizer of $\exp(2\pi i\lambda)$ in $G$.
Then Conjecture \ref{fgt'} says that the categories $D(B_{\lambda}\backslash G/B_{\lambda})$ and $D((B^{\vee},\lambda)\dashedbackslash \, G^{\vee}\dashedslash (B^{\vee},\lambda))$ are Koszul dual where the latter stands for the corresponding monodromic version of the category $D(B^{\vee}\dashedbackslash \, G^{\vee}\dashedslash B^{\vee})$.
This statement in fact follows from the main results of \cite{by} and \cite{ly}. It should not be difficult to extend the arguments of \S\ref{hecke} to show that the equivalence of the
corresponding $\BZ/2$-graded categories we get from Theorem \ref{main'} is the same as the equivalence one gets by combining the results of \cite{by} and \cite{ly}, but we shall not pursue it in this paper.
\end{ex}

\subsection{Idea of the proof}
Let us very briefly outline the basic idea of the proof. We start with the equivalence of Conjecture \ref{bzsv} and choose a regular
coweight $\lambda\in\ft$ of $G$. (Here $\ft=\on{Lie}\sT$ is a Cartan Lie subalgebra in $\fg=\on{Lie}G$.)
After passing to the corresponding $\BZ/2$-graded categories everything becomes linear over
$\Sym(\ft^*\oplus \BC)$ and we can perform the following procedures to both sides of the equivalence of Conjecture \ref{bzsv}:

1) Completion at $a=(\lambda,1)$.

 2) Restriction to the point 1 in the second summand of $\gra=\ft\oplus \BC$. We get a category linear over $\Sym(\ft^*)$).

 3) Shift by $\lambda$ in $\ft$ (i.e.\ this is just the identity operation on the level of abstract categories, but we change the action of $\Sym(\ft^*)$ via the pullback under the automorphism $x\mapsto x+\lambda$ of $\ft$).

\smallskip
\noindent
We then check that the above procedures applies to the LHS and the RHS of the $\BZ/2$-version of the equivalence of Conjecture \ref{bzsv} lead precisely to the LHS and the RHS of the equivalence of Theorem \ref{main}.
The other statements announced above are proven by similar techniques.

One remark is in order: a priori this construction depends on the choice of $\lambda$. However, we show in~\S\ref{indep}
that the resulting equivalence is in fact independent of this choice.

We should note that our approach is very similar to the arguments of~\cite{bzn}, although with some differences.

\subsection{Acknowledgments}
We are grateful to R.~Bezrukavnikov, D.~Nadler, G.~Shtotland, D.~Timashev, R.~Yang and Z.~Yun for the useful discussions.
Above all we are indebted to V.~Ginzburg. He decided not to sign this work in the author's capacity, though it is clear
that our note is a direct outgrowth of his fundamental work~\cite{bgs} and our previous paper~\cite{fgt}. A.B.\ was partially supported by NSERC.
The research of M.F.\ was supported by the Israel Science Foundation (grant No.~994/24) and BSF (grant No.~3013008329).


\section{Generalities on completions of equivariant derived categories}

\subsection{Completion of a dg-category}
Let $\calC$ be a cocomplete dg category which is linear over a commutative regular noetherian ring $R$
(i.e.\ it is enriched over the corresponding dg category of $R$-modules).
For any $x\in \Spec(R)$ let $\hatR_x$ be the completion of $R$ at $x$. We consider the category $R\mod^\wedge_x$ which can be described as
the cocompletion of the category of coherent objects of $R$-mod, which are set-theoretically supported at the
closure of $x$ (usually we will be applying this to a closed point $x$).
We then define $\calC^\wedge_x =\calC\underset{R\mod}\otimes R\mod^\wedge_x$. We have a natural functor from $\calC$ to
$\calC^\wedge_x$. If in addition we assume that $\calC$ is compactly generated and that Hom's between compact objects are
compact objects of $R$-mod (which in our case is the same as coherent complexes), then for two compact objects $X$ and $Y$ the
space of Hom's between their images in $\calC^\wedge_x$ is equal to $\Hom(X,Y)\underset{R}\otimes \hatR_x$ (warning: the image of
compact objects in the completion is usually not compact).

Similar (somewhat easier) constructions work for the fiber $\calC_x$ of $\calC$ at $x$. Also the constructions can be carried over in the world of differential $\BZ/2$-graded categories instead
of dg categories.

\subsection{$\BZ/2$-graded equivariant derived categories and their completions}
\label{completion}
Let $A$ be a complex algebraic torus acting on an algebraic variety $Y$ over $\mathbb C$. Abusing slightly the notation we let
$\calC$ as above be the $\BZ/2$-graded renormalized equivariant derived category
$(\tilD(Y)^A)^{\BZ/2}$, i.e.\ the 2-periodization of the renormalized equivariant derived category.
(More precisely we should work with the corresponding dg category and its $\BZ/2$-graded version but we are
going to ignore this in the notation.)
Then this category is enriched over $R$-mod where $R=H^\bullet_A(\pt)=\Sym(\gra^*)$ where $\gra$ denotes the Lie algebra of $A$.
Thus for any $a\in \gra$ we can consider the completed
derived category $((D(Y)^A)^{\BZ/2})^\wedge_a$. It is easy to see that if we take $a=0$, then the corresponding completion is just
$(D(Y)^A)^{\BZ/2}$.\footnote{This statement makes sense and it is true also for the corresponding $\ZZ$-graded categories.}

\begin{ex}
  Let $Y=\mathrm{pt}$. Then $(D(Y)^A)^{\BZ/2}$ is the $\BZ/2$-derived category of $\Sym(\gra^*)$-modules supported at $0$, while
  $(\tilD(Y)^A)^{\BZ/2}$ is the corresponding category of all $\Sym(\gra^*)$-modules.
Similarly, $\tilD(Y)^A$ is the derived category of all $\Sym(\gra^*)$ dg-modules, while $D(Y)^A$ is the ind-completion of the derived category of finite-dimensional graded modules.
\end{ex}

\begin{thm}\label{loc}
Let $Y^a$ denote the fixed points of $a$ on $Y$ and let $i\colon Y^a\to Y$ be the natural inclusion.
Then the following holds:
\begin{enumerate}
\item
  The category $((\tilD(Y^a)^A)^{\BZ/2})^\wedge_a$ is naturally equivalent to $((\tilD(Y^a)^A)^{\BZ/2})^\wedge_0=(D(Y^a)^A)^{\BZ/2}$
  after applying shift by $a$ in the Lie algebra of $\gra$ (i.e.\ we have an equivalence of abstract categories, and the two
  actions of $\Sym(\fa^*)$ differ by shift by $a$, i.e.\ by the pullback under the automorphism $x\mapsto x+a$ of $\fa$).
\item
The natural morphism $i^!\to i^*$ becomes an isomorphism after completion at $a$.
\item
The resulting functor $((\tilD(Y)^A)^{\BZ/2})^\wedge_a\to (D(Y^a)^A)^{\BZ/2}$ is a full embedding.
\item
  The same is true when completion at $a$ is replaced with fiber at $a$ and $(D(Y^a)^A)^{\BZ/2}$ is replaced by $(D(Y)^{A_\mon})^{\BZ/2}$.
  In particular $(\tilD(Y)^A)^{\BZ/2}_a$ is a full subcategory of $(D(Y^a)^{A_\mon})^{\BZ/2}$.
\end{enumerate}
\end{thm}
The proof is standard. Let us comment on (2) since we are going to need later the following generalization of it:
\begin{lem}
Let $\eta\colon Z\to Y$ be locally closed embedding of an $A$-invariant subvariety such that $Z$ contains $Y^a$. Then after completion at $a$ the functors
$\eta^*$ and $\eta^!$ are naturally isomorphic.
\end{lem}
\begin{proof}
  By replacing $Z$ with its closure we can assume that $Z$ is closed (since $Z$ is open in its closure and since for an open embedding the
  two restriction functors obviously coincide). Let $U=Y\backslash Z$ and let $j$ be the embedding of $U$ into $Y$. It is easy to see that
  since $U^a=\emptyset$ it follows that $\tilD(U)^A$ vanishes after completion at $a$. Hence for every $\calF\in \tilD(U)^A$ the image of
  both $j_*\calF$ and $j_!\calF$ in $(\tilD(Y)^A)^\wedge_a$ is~$0$. This shows that for every $\calG\in \tilD(Y)^A$ the maps
  $\eta_!\eta^!\calG\to \calG$ and $\calG\to \eta_*\eta^*\calG$ become isomorphisms after completion at $a$, which shows that $\eta^!$
  and $\eta^*$ become isomorphic after completion at $a$.
\end{proof}

Now let us choose a one-parametric subgroup $\BG_m\to A$ such that $Y^{\BG_m}=Y^a$.
\begin{cor}\label{braden}
  Let $\Phi\colon \tilD(Y)^A\to \tilD(Y^a)^A$ denote the hyperbolic restriction functor with respect to the above
  action of $\BG_m$ (we recall the definition below) and let $i$ as before be the embedding of $Y^a$ into $Y$.
  Then $\Phi$ becomes isomorphic to $i^*$ and $i^!$ after completion at $a$.
\end{cor}
\begin{proof}
  Let $Z\subset Y$ be the attractor of $\BG_m$ to $Y^a$. The hyperbolic restriction is defined as $!$-restriction to
  $Z$ followed by $*$-restriction to $Y^a$. Since $Z$ contains $Y^a$ we can replace the first $!$-restriction by
  $*$-restriction after completion at $a$ and the corollary follows.
\end{proof}

\subsection{A variant}
\label{variant}
Assume now that $Y$ is acted on by a reductive group $\GG$ with maximal torus $A$.\footnote{In applications $\GG$ will usually be $G\times \BC^{\times}$.}
Let $a$ be a regular element of $\gra$. We have a natural functor $(\tilD(Y)^{\GG})^{\BZ/2}\to (\tilD(Y)^A)^{\BZ/2}$; the first category is linear over $\Sym(\gra^*)^W$ where $W$ is the Weyl group of
$\GG$, while the second is linear over $\Sym(\gra^*)$. Hence it makes sense to take the completion of both at our $a\in \gra$. It is easy to see that the above functor induces an equivalence
between $((D(Y)^\GG)^{\BZ/2})^\wedge_a$ and $((D(Y)^A)^{\BZ/2})^\wedge_a$.\footnote{It is important here that $a$ is assumed to be regular.}
Similar statement holds when completion at $a$ is replaced by fiber at $a$.

\section{Proof of Theorem \ref{main}}\label{idea}
\subsection{Idea of the proof}
\label{idea of the proof}
To simplify the discussion we shall only deal with the untwisted case; the twisted case is obtained similarly. Also, let us concentrate on proving the first assertion of Theorem \ref{main}. The 2nd assertion is proved completely analogously if one replaces all completions below by the corresponding fibers.

The idea of the proof is the following. Let us pick a regular dominant integral coweight $\lambda$ of $G$. Let $A=\sT\times \BC^{\times}$.
 We would like to take the equivalence of Conjecture \ref{bzsv} and apply to both sides the composition of the following~3 operations:

 1) Completion at $a=(\lambda,1)$.

 2) Restriction to the point 1 in the second summand of $\gra=\grt\oplus \BC$. We get a category linear over $\Sym(\grt^*)$.

 3) Shift by $\lambda$ in $\grt$ (i.e.\ this is just the identity operation on the level of abstract categories, but we change the action of $\Sym(\grt^*)$ via the pullback under the automorphism $x\mapsto x+\lambda$ of $\ft$).

\smallskip
\noindent
We claim that after applying these~3 functors to both sides of Conjecture \ref{bzsv} we get the two sides in Theorem \ref{main}(1).
Let us explain this for the right hand side. First, it is clear that the first two operations commute.
Applying~2) to the category $((Q_{\hbar}(\fX^{\vee})\mod)^{G^{\vee}})^{\BZ/2}$, we get the category
$((Q_{\hbar=1}(\fX^{\vee})\mod)^{G^{\vee}})^{\BZ/2}$. But weakly $G^\vee$-equivariant modules over the quantization
$Q_{\hbar=1}(\fX^\vee)$ are nothing but strongly $G^\vee$-equivariant modules over $Q_{\hbar=1}(\fX^\vee)\otimes U_{\hbar=1}(\fg^\svee)$.
Namely, given a weakly $G^\vee$-equivariant $Q_{\hbar=1}(\fX^\vee)$-module $M$, there are two $\fg^\svee$-actions on $M$.
One $b_1(x)$, $x\in\fg^\svee$, is the differential of the $G^\vee$-action on $M$. Another $b_0(x)$ comes from the homomorphism
$U_{\hbar=1}(\fg^\svee)\to Q_{\hbar=1}(\fX^\vee)\to\End(M)$. Then $\beta(x):=b_1(x)-b_0(x)$ gives rise to a homomorphism from
$U_{\hbar=1}(\fg^\svee)$ to $\End_{Q_{\hbar=1}}(M)$.

The center $ZU_{\hbar=1}(\fg^\svee)$ acts on the category
$((Q_{\hbar=1}(\fX^{\vee})\mod)^{G^{\vee}})^{\BZ/2}\cong\big(((Q_{\hbar=1}(\fX^\vee)\otimes U_{\hbar=1}(\fg^\svee))\mod)^{G^{\vee},\on{strong}}\big)^{\BZ/2}$
via the embedding into the 2nd factor.
Using the derived Beilinson-Bernstein localization theorem we identify the completion at $\lambda$ of
$((Q_{\hbar=1}(\fX^{\vee})\mod)^{G^{\vee}})^{\BZ/2}\cong\big(((Q_{\hbar=1}(\fX^\vee)\otimes U_{\hbar=1}(\fg^\svee))\mod)^{G^{\vee},\on{strong}}\big)^{\BZ/2}$
with the category
$\big(((Q_{\hbar=1}(\fX^\vee)\otimes\CalD(G^\vee/N^\vee)\mod^{\widehat{\Lambda}})^{G^{\vee},\on{strong}}\big)^{\BZ/2}$
of strongly $G^\vee$-equivariant modules over $Q_{\hbar=1}(\fX^\vee)\otimes\CalD(G^\vee/N^\vee)$ that are unipotent monodromic
on the 2nd factor.

In more detail, we consider the category of $\CalD(G^\vee/N^\vee)$-modules
(sheaves of modules over differential operators on $G^\vee/N^\vee$)
such that the action of $\BC[\ft]$ (differential operators arising from the infinitesimal right action of $\sT^\vee$ on $G^\vee/B^\vee$)
has generalized eigenvalues in the weight lattice $\Lambda$ of $\sT^\vee$.
We denote this category by $\CalD(G^\vee/N^\vee)\mod^{\widehat{\Lambda}}$.
For $M\in\CalD(G^\vee/N^\vee)\mod^{\widehat{\Lambda}}$, we denote by $R\Gamma^{\widehat{\lambda}}(G^\vee/N^\vee,M)$
the maximal direct summand of the global sections on which
$\BC[\ft]$ acts (via infinitesimal right action of $\sT^\vee$ on $G^\vee/N^\vee$) with the generalized eigenvalue $\lambda$.
The functor
\begin{equation}
  \label{domi}
  R\Gamma^{\widehat{\lambda}}(G^\vee/N^\vee,?)\colon
  \CalD(G^\vee/N^\vee)\mod^{\widehat{\Lambda}}\to U_{\hbar=1}(\fg^\svee)\mod^{\widehat{\lambda}}
\end{equation}
is a $t$-exact equivalence of categories since $\lambda+\rho$ is regular dominant.

\begin{rem} Taking the fiber (as opposed to taking completion) at $\lambda$ corresponds
to taking $B^\vee$-equivariant modules over $Q_{\hbar=1}(\fX^\vee)$ with the same proof -- this is needed in order to prove Theorem \ref{main}(2).
\end{rem}

\begin{rem} A priori in this way we construct an equivalence claimed by Theorem \ref{main} which depends on the choice of $\lambda$.
  However, we will show in~\S\ref{indep} that it is in fact independent of $\lambda$.
\end{rem}

Thus it remains to prove the following:
\begin{thm}\label{B-equiv}
  The composition of 1, 2 and 3 above applied to the category $(\tilD(\fY(\CK))^{G(\CO)\rtimes \mathbb C^\times})^{\BZ/2}$ is equivalent to
  $(D(\fY)^B)^{\BZ/2}$.
\end{thm}
\subsection{Proof of Theorem \ref{B-equiv}}
To prove Theorem \ref{B-equiv} we need two preparatory lemmas.
Let $\fY$ be a spherical $G$-variety, and let $\fX=T^*_\psi\fY$ be its twisted cotangent bundle. If $\fX$ is
hyperspherical~\cite[\S3.5.1]{bzsv}, then according to~\cite[Proposition 3.7.4.(b)]{bzsv}, the $B$-stabilizers of
points in the open $B$-orbit in $\fY$ are connected.

\begin{lem}[D.~Timashev]
  \label{tim}
  If the $B$-stabilizers in the open $B$-orbit in a spherical $G$-variety $\fY$ are connected, then the
  $B$-stabilizer of any point $y\in\fY$ is connected.
\end{lem}

\begin{proof}
First we consider the closure $\fY'$ of a $G$-orbit in $\fY$, and assume that the $B$-stabilizer of a general
point in $\fY'$ is connected. Then it follows that the $B$-stabilizer of any point $y'\in\fY'$ is connected.
Indeed, according to~\cite[Corollary 3.3]{k2}, the group of connected components $\on{Stab}_B(y')/\on{Stab}_B^0(y')$
is a subquotient of the similar group $\on{Stab}_B(y'')/\on{Stab}_B^0(y'')$ for $y''\in\fY'$ lying in a bigger
$B$-orbit. By induction in the set of $B$-orbits in $\fY'$ with its adjacency order, we see that
$\on{Stab}_B(y')/\on{Stab}_B^0(y')$ is a subquotient of $\on{Stab}_B(y'')/\on{Stab}_B^0(y'')$ for $y''$ in
the open $B$-orbit in $\fY'$. But the latter group is trivial by our assumption.

It remains to study the $B$-stabilizers of points in $B$-orbits open in their $G$-orbits in $\fY$. For a point
$y$ in such a $B$-orbit, the group $\on{Stab}_B(y)/\on{Stab}_B^0(y)$ is trivial iff the lattice $\Lambda_{\fY'}$
is saturated in $\Lambda$. Here $\Lambda$ is the lattice of weights of $B$ (i.e.\ $\Lambda=X^*(\sT)$),
and $\Lambda_{\fY'}$ is the lattice of weights of $B$-semi-invariant rational functions on the $G$-orbit
$\fY'=G.y$. According to~\cite[Theorem 1.3.(b)]{k1}, any $B$-semi-invariant rational fuction on $\fY'$
extends to a $B$-semi-invariant rational function on the whole of $\fY$. Clearly, the $B$-weights of the function
and its extension coincide. The resulting rational functions on $\fY$ can be characterized as $B$-semi-invariant
rational functions without poles or zeros at the $B$-invariant irreducible divisors in $\fY$ containing $\fY'$.
It follows that $\Lambda_{\fY'}$ is saturated in the lattice of $B$-weights of $B$-semi-invariant rational
functions on $\fY$. Hence the triviality of $\on{Stab}_B(y)/\on{Stab}_B^0(y)$ for general $y\in\fY'$ follows
from the triviality of $\on{Stab}_B(y)/\on{Stab}_B^0(y)$ for general $y\in\fY$.

The lemma is proved.
\end{proof}

\begin{lem}
  \label{B-orbit}
  \begin{enumerate}
  \item Every $G(\CO)\rtimes \mathbb C^\times$-orbit in $\fY(K)$ that intersects $t^\mu\cdot\fY$ for some integral coweight
    $\lambda$ is in fact a single $G(\CO)$-orbit.
  \item
  Given a $G(\CO)$-orbit $\BO$ in $\fY(\CK)$, for a regular antidominant coweight $\mu$,
  the intersection $\BO\cap(t^\mu\cdot\fY)$ is a disjoint union of $B$-orbits in $\fY$.
  \end{enumerate}
\end{lem}

\begin{proof}
Let $H=F\rtimes T$ be a semidirect product of a torus with an algebraic group. Let $X$ be an
$H$-homogeneous space. Then any connected component $X^T_i$ of the fixed point set $X^T$ is homogeneous
with respect to the action of $(H^T)^\circ$ (the neutral connected component of the fixed point
subgroup of the conjugation $T$-action on $H$).

Indeed, the tangent space $T_xX^T_i$ coincides with the tangent space
$T_x\left((H^T)^\circ.x\right)$ of
the $(H^T)^\circ$-orbit through $x$. All the tangent spaces $T_xX^T_i$, $x\in X^T_i$, are of the
same dimension, hence all the $(H^T)^\circ$-orbits in $X^T_i$ are of the same dimension, hence
$X^T_i$ must have a unique $(H^T)^\circ$-orbit (being connected).

We will apply the above claim to $F=G(\CO)$, and
$T=\BC^\times$ embedded via $(\mu,1)$ into the product of loop rotations with the Cartan torus of
$G$. Then $H^T=B\rtimes\BC^\times$. Furthermore, $X$ is a $G(\CO)$-orbit $\BO$ in $\fY(\CK)$. Then
$X^T=\BO\cap(t^\mu\cdot\fY)$.
\end{proof}

\begin{rem}
  It might be tempting to conjecture that if $\lambda$ is large enough then the intersection of a $G(\CO)$-orbit in
  $\fY(\CK)$ with $t^{-\lambda}\cdot\fY$ is always a single $B$-orbit (i.e.\ that for sufficiently large
  $\lambda$ one does not need to pass to connected components).
  However, this is wrong in general, as the following example shows.

  Consider the diagonal example $G=\SL(3)\times\SL(3)$, $\fY=\SL(3)$. The Weyl group $W=S_3$ is generated
  by simple reflections $s_1,s_2$. Set $\Gr=\SL(3,\CK)/\SL(3,\CO)$.
  One could expect the following. Let $\mu,\nu$ be big enough dominant regular
  coweights of $\SL(3)$. Let $\CB^\mu\subset\Gr$ be the $\SL(3)$-orbit of $t^\mu$ (isomorphic to
  the full flag variety of $\SL(3)$). Then the intersection of the shifted orbit $t^{-\nu}\CB^\mu$
  with any $\SL(3,\CO)$-orbit $\Gr^\theta\subset\Gr$ is exactly one Schubert cell in $t^{-\nu}\CB^\mu$.

  However, if we take $\mu=\nu=N\rho=(N,0,-N)=\theta$ for $N\gg0$, then the $\sT$-fixed points $t^{-\nu}\cdot t^{s_1\mu}$
  and $t^{-\nu}\cdot t^{s_2\mu}$ both lie in $\Gr^\theta$. Hence $t^{-\nu}\CB^\mu\cap\Gr^\theta$ consists
  of {\em two} Schubert cells. Note however that $t^{-\nu}\CB^\mu\cap\Gr^\theta$ is disconnected (a {\em disjoint}
  union of two Schubert cells).
\end{rem}

We can now finish the proof of Theorem \ref{B-equiv} (and hence the proof of Theorem~\ref{main}). As the first step, we consider
the completion of $(D(\fY(\CK))^{G(\CO)\rtimes\BC^\times})^{\BZ/2}$ at $(\lambda,1)$.
As $\lambda$ is assumed regular,
by Theorem \ref{loc} (and \S\ref{variant}) this completion is a full subcategory of
$(D(\fY(\CK)^{(t^\lambda,t)})^{\sT\times\BC^\times})^{\BZ/2}$ (where $\fY(\CK)^{(t^\lambda,t)}$ means fixed points with respect to the copy of
$\BC^{\times}$ in $\sT\times\BC^{\times}$ consisting of pairs $(t^{\lambda},t)$). However, this scheme of fixed points is
isomorphic to $\fY$ under the map $y\mapsto t^{-\lambda} y$ (i.e.\ the fixed points are equal to
$t^{-\lambda}\cdot\fY\subset \fY(\CK)$). It follows now from the combination of~Lemma~\ref{tim} and~Lemma~\ref{B-orbit}
that the above full subcategory is equal to $(D(\fY)^{B\rtimes\BC^{\times}})^{\BZ/2}$, where $\BC^{\times}$ acts on $\fY$ by the
cocharacter $-\lambda$ of $\sT$. It is easy to see that in this case, specializing $\hbar$ to 1 just gives the shift of
the category $(D(\fY)^B)^{\BZ/2}$ by $-\lambda$ over $\Sym(\ft^*)$. Shifting back by $\lambda$ gives the category
$(D(\fY)^B)^{\BZ/2}$.

\subsection{Proof of Theorem \ref{main'}}
The proof follows the same pattern as the proof of Theorem \ref{main}: we still apply the same procedures 1, 2 above for
$a=(\lambda,1)$.
In the right hand side we get the category $(Q_{\hbar=1}(\fX^{\vee})\mod^{B^{\vee}_{\mon,\lambda}})^{\BZ/2}$ for the same reason as before. Thus we need to prove an analog of Theorem \ref{B-equiv}. For this it is enough to identify the fixed points of
$a$ on $\fY(\calK)$ with $\fY^{\sT_\lambda}$ (and check the analog of Lemma \ref{B-orbit} which will be again a word-by-word repetition of the original case).

We consider the universal cover $\alpha\to \exp(2\pi \sqrt{-1} \alpha)$ of the multiplicative group.
For any $\lambda\in \grt$ and any point $y$ in $\fY$, an ``additive loop"
$\exp (-2\pi \sqrt{-1} \alpha \lambda)\cdot y$ is fixed under the action of the element $a\in \grt\oplus\BC$ for evident reasons.
The only problem is to make sure that the above ``additive loop" is an actual loop (i.e.\ factors through the above universal cover).
This happens if and only if the point $y$ is fixed under the element $\exp(2\pi \sqrt{-1} \lambda)$ of $\sT$ (i.e.\ under the Zariski
closure of the subgroup generated by this element). This way we get a map $\fY^{\sT_{\lambda}}\to \fY(\CK)^a$, and it is easy to see that
this map is an isomorphism.

\subsection{Examples}
Here we discuss some examples where Theorem \ref{bzsv} is known and thus our Theorem \ref{main} becomes unconditional.

\subsubsection{Jacquet-Shalika}\label{js}
$\fX=T^*(\GL(2n)/\Sp(2n))\circlearrowleft \GL(2n)$, and the dual variety
$\fX^\vee=T^*_\psi(\GL(2n)/(\GL(n)_{\mathrm{diag}}\ltimes\fgl(n)))$. In this example, Conjecture \ref{bzsv} was
proved in~\cite{cmno}.

\subsubsection{Rankin-Selberg, or mirabolic}\label{mir}
$\fX=T^*(\BC^N\times\GL(N))\circlearrowleft(\GL(N)\times\GL(N))$, and
$\fX^\vee=T^*\End(\BC^N)\circlearrowleft(\GL(N)\times\GL(N))$. In this example, Conjecture \ref{bzsv} was
proved in~\cite{bfgt}.

\subsubsection{$\GL(M|N)$}\label{glmn}
$\fX$ is the equivariant slice to a nilpotent orbit in $\fgl(N)$ of Jordan type $(N-M,1^M)$, acted upon by
$\GL(N)\times\GL(M)$. In fact, $\fX$ is the twisted cotangent bundle $T^*_\psi(\GL(N)/U_{M,N})$, and
$\fX^\vee=T^*\Hom(\BC^N,\BC^M)$. In this example, Conjecture \ref{bzsv} was proved in~\cite{bfgt} for $M=N-1$ and in~\cite{ty} in general.

\subsubsection{Bessel}\label{bessel}
First, we can take $\fX=T^*(\SO(2n+2))\circlearrowleft(\SO(2n+2)\times\SO(2n+1))$, and
$\fX^\vee=(\BC^{2n+2}_+\otimes\BC^{2n}_-)\circlearrowleft(\SO(2n+2)\times\Sp(2n))$.
Second, we can take $\fX=T^*(\SO(2n+1))\circlearrowleft(\SO(2n+1)\times\SO(2n))$, and
$\fX^\vee=(\BC^{2n}_+\otimes\BC^{2n}_-)\circlearrowleft(\SO(2n)\times\Sp(2n))$.
In these examples, Conjecture \ref{bzsv} was proved in~\cite{bft1}.

\subsubsection{Symplectic mirabolic}
The following is not formally an example to which the results of the present paper are directly applicable.
However, we hope that some upgrade of those results will also prove an analogue of Theorem \ref{main} in this case.

Let $\fX=(\BC^{2n}_-\times T^*\Sp(2n))\circlearrowleft(\Sp(2n)\times\Sp(2n))$, and
$\fX^\vee=(\BC^{2n+1}_+\otimes\BC^{2n}_-)\circlearrowleft(\SO(2n+1)\times\Sp(2n))$.
In this example, Conjecture \ref{bzsv} was proved in~\cite{bft2}. Note however that this example is not
polarizable (and also anomalous, so that the dual of the second copy of $\Sp(2n)$ is the {\em metaplectic}
dual).

\section{Compatibility with Hecke action}\label{hecke}

\subsection{The diagonal case}
\label{diagonal case}
Recall the setup of~\S\ref{idea of the proof} and consider the diagonal case where $\fY=G\circlearrowleft\bG=G\times G$, and
$\fX^\vee=T^*(G^\vee)\circlearrowleft\bG^\vee=G^\vee\times G^\vee$. We consider the Cartan torus $\bT^\vee=\sT\times\sT\subset\bG^\vee$
and its dominant integral weight $\blambda=(\lambda,-w_0\lambda)$. Using the Beilinson-Bernstein localization theorem for the functor
\[R\Gamma^{\widehat{\lambda,-w_0\lambda}}\colon \CalD((G^\vee/N^\vee)^2)\mod^{\widehat{\Lambda\times\Lambda}}\to
U_{\hbar=1}(\fg^\svee)\otimes U_{\hbar=1}(\fg^\svee)\mod^{\widehat{\lambda,-w_0\lambda}},\]
we identify the completion at $\blambda$ of $((Q_{\hbar=1}(\fX^\vee)\mod)^{\bG^\vee})^{\BZ/2}$ with $\CalD$-module category
$\big(((\CalD(G^\vee)\otimes\CalD(G^\vee/N^\vee)\otimes\CalD(G^\vee/N^\vee)\mod^{\widehat{\Lambda\times\Lambda}})^{G^{\vee},\on{strong}}\big)^{\BZ/2}$.

Note that
$\big((\CalD(G^\vee)\otimes\CalD(G^\vee/N^\vee)\otimes\CalD(G^\vee/N^\vee)\mod^{\widehat{\Lambda\times\Lambda}})^{G^{\vee},\on{strong}}\big)^{\BZ/2}$
is equivalent to $D(B^{\vee}\dashedbackslash \, G^{\vee}\dashedslash B^{\vee})^{\BZ/2}$.
Namely, $G^\vee\times G^\vee\times \sT^\vee\times\sT^\vee$ acts on $G^\vee\times(G^\vee/N^\vee)\times(G^\vee/N^\vee)$ as follows:
\((g_1,g_2,t_1,t_2)(g,[g'],[g''])=(g_1gg_2^{-1}, [g_1g't_1], [g_2g''t_2]),\) where $[g]$ denotes the class of $g$ in $G^\vee/N^\vee$.
If we restrict to $G^\vee\times[1]\times[1]$, we obtain $G^\vee$ with the following residual action of
$N^\vee\times N^\vee\times\sT^\vee\times\sT^\vee$:
\[(n_1,n_2,t_1,t_2)(g)=n_1t_1^{-1}gt_2n_2^{-1}.\]

On the other hand, if we restrict to $1\times(G^\vee/N^\vee)\times(G^\vee/N^\vee)$, we obtain an equivalence of
$\big((\CalD(G^\vee)\otimes\CalD(G^\vee/N^\vee)\otimes\CalD(G^\vee/N^\vee)\mod^{\widehat{\Lambda\times\Lambda}})^{G^{\vee},\on{strong}}\big)^{\BZ/2}$
with $(D\CM^{\widehat{\Lambda},\widehat{\Lambda}})^{\BZ/2}$: $G^\vee$-equivariant derived category of
$\CalD(G^\vee/N^\vee)^2$-modules such that the action of $\BC[\ft]$ (differential operators arising from the infinitesimal right action
of $\sT^\vee$ on the first (resp.\ second) copy of $G^\vee/N^\vee$)) has generalized eigenvalues in $\Lambda$.
The action of $G^\vee\times\sT^\vee\times\sT^\vee$ is as follows:
\[(g,t_1,t_2)([g_1],[g_2])=([gg_1t_1],[gg_2t_2]).\]

\subsection{Intertwining functors and localization theorem}\label{intertwining}
As we have seen in~\S\ref{diagonal case}, the category $D(B^{\vee}\dashedbackslash \, G^{\vee}\dashedslash B^{\vee})$ is equivalent to
the category $D\CM^{\widehat{\Lambda\times\Lambda}}$: $G^\vee$-equivariant derived category of
$\CalD((G^\vee/N^\vee)^2)$-modules such that the action of $\BC[\ft]$ (differential operators arising from the infinitesimal right action
of $\sT^\vee$ on the first (resp.\ second) copy of $G^\vee/N^\vee$)) has generalized eigenvalues in $\Lambda$. We have a monoidal equivalence
$R\Gamma^{\widehat{\lambda,-\lambda-2\rho}}\colon D\CM^{\widehat{\Lambda\times\Lambda}}\iso\HC(\fg^\svee)^{\widehat\lambda}$:
the derived category of Harish-Chandra bimodules over $U(\fg^\svee)$ with generalized central character $\lambda$ on both sides,
see~\cite[\S2.3]{bfo}. 
We denote by $\Psi_\lambda$ the composition
\[D(B\backslash G/B)^{\BZ/2}\iso D(B^{\vee}\dashedbackslash \, G^{\vee}\dashedslash B^{\vee})^{\BZ/2}\cong
(D\CM^{\widehat{\Lambda\times\Lambda}})^{\BZ/2}\iso (\HC(\fg^\vee)^{\widehat\lambda})^{\BZ/2}.\]

Note that the functor $R\Gamma^{\widehat{\lambda,-\lambda-2\rho}}$ is {\em not} a particular case of a functor in~\eqref{domi}.
In our present notation, the group considered in~\S\ref{idea of the proof} is rather $G^\vee\times G^\vee$, and
the weight considered in {\em loc.cit.}\ is a pair of weights $(\mu,\nu)$. But in~\eqref{domi} the weight was dominant, i.e.\ both
$\mu$ and $\nu$ were dominant, whereas in $R\Gamma^{\widehat{\lambda,-\lambda-2\rho}}$ we have $\mu=\lambda$, $\nu=-\lambda-2\rho$.
The functors of {\em loc.cit.}\ and of the present section are related by the following intertwining functor.

We have the long intertwining functor
$I_{w_0}^{-1}\colon D\CM^{\widehat{\Lambda\times\Lambda}}\iso D\CM^{\widehat{\Lambda\times\Lambda}}$
such that \[R\Gamma^{\widehat{\lambda,-\lambda-2\rho}}\circ I_{w_0}^{-1}\simeq R\Gamma^{\widehat{\lambda,-w_0\lambda}},\]
see~\cite[\S3.2]{bfo}. Namely, let ${\mathbb V}\subset(G^\vee/N^\vee)^4=\{(x_1,x_2,y_1,y_2) : x_1=y_1,\ (x_2,y_2)$ lies in the open
$G^\vee\times\sT^\vee$ orbit in $(G^\vee/N^\vee)^2\}$. Let $\pr_x,\pr_y\colon{\mathbb V}\to(G^\vee/N^\vee)$ send $(x_1,x_2,y_1,y_2)$ to
$(x_1,x_2)$, $(y_1,y_2)$ respectively. Then $I_{w_0}^{-1}=\pr_{y*}\pr_x^!$.

\subsection{A short digression on \cite{by}}
The purpose of this Section is to prove Theorem \ref{byun}. In order to do this we need a characterization of the equivalence between
$D(B\backslash G/B)^{\BZ/2}$ and  $D(B^{\vee}\dashedbackslash \, G^{\vee}\dashedslash B^{\vee})^{\BZ/2}$ coming from \cite{by}. First, we
identify the RHS with the ($\BZ/2$-version of) the derived category $(\HC(\fg^\svee)^{\widehat\lambda})^{\BZ/2}$ of Harish-Chandra bimodules over
$U(\fg^\svee)$ with generalized central character $\lambda$ on both sides, where $\lambda$ is a regular integral coweight of $\fg$.
Then we have the following
\begin{lem}\label{lem-byun}
\begin{enumerate}
\item
  The equivalence between $D(B\backslash G/B)^{\BZ/2}$ and $(\HC(\fg^\svee)^{\widehat\lambda})^{\BZ/2}$ coming from \cite{by}
  satisfies the following properties:

a) It is monoidal.

b) It is linear over $\Sym(\ft^*\oplus\ft^*)$.

c) It sends (the image in $D(B\backslash G/B)^{\BZ/2}$ of) every irreducible perverse sheaf in $D(B\backslash G/B)$ to
a projective Harish-Chandra bimodule.

\item The above equivalence is uniquely characterized by properties a), b) and~c).
\end{enumerate}
\end{lem}
\begin{proof}
  Let us prove the 1st assertion. The properties a) and b) are immediate from \cite{by}. The Bezrukavnikov-Yun Koszul duality when
  viewed as an equivalence between $D(B\backslash G/B)^{\BZ/2}$ and $D(B^{\vee}\dashedbackslash \, G^{\vee}\dashedslash B^{\vee})^{\BZ/2}$
  sends the images of irreducible perverse sheaves to images of tilting objects from
  $D(B^{\vee}\dashedbackslash \, G^{\vee}\dashedslash B^{\vee})$. On the other hand, it is explained in~\cite{bg}
  (cf.\ also~\cite[\S2]{bfo}) that the natural monoidal equivalence between $D(B^{\vee}\dashedbackslash \, G^{\vee}\dashedslash B^{\vee})$
  and $\HC(\fg^\svee)^{\widehat\lambda}$ sends tilting objects to projective objects.

  Let us now prove the 2nd assertion. Assume that we are given an equivalence satisfying a), b) and c). Composing it with the inverse of the
  Bezrukavnikov-Yun equivalence, we get an auto-equivalence $F$ of the category $D(B\backslash G/B)^{\BZ/2}$ which satisfies properties~a)
  and~b), and in addition sends (the images of) irreducible pervese sheaves to (the images of)  semi-simple perverse sheaves and hence to
  (the images of) irreducible perverse sheaves.
We claim that such an auto-equivalence is isomorphic to the identity functor. First, every irreducible perverse sheaf
goes to itself (i.e.\ we can choose an isomorphism between $\calF$ and $F(\calF)$ for every simple perverse sheaf
$\calF$).\footnote{From now on we shall abuse the language by not writing ``image in  $D(B\backslash G/B)^{\BZ/2}$"
for objects of $D(B\backslash G/B)$.} Indeed, for every $w\in W$ consider $\text{RHom}_{B\times B}(\IC_w,\IC_w)$,
where $\IC_w$ denotes the corresponding irreducible perverse sheaf on $B\backslash G/B$. The above RHom is a bimodule
over $\Sym(\ft^*)$ which is supported on the union of graphs of all $w'$ where $w'\leq w$. Hence these supports are
different for different $w$. It follows that any $\Sym(\ft^*\oplus\ft^*)$-linear autoequivalence which permutes the
$\IC_w$'s, must send every $\IC_w$ to itself.

The rest of the proof is essentially a word-by-word repetition of the proof of~\cite[Proposition in \S2.5]{bbm}.
\end{proof}

\subsection{Proof of Theorem \ref{byun}} To prove~Theorem~\ref{byun} it is enough to prove that the equivalence $\Psi_{\lambda}$
of~\S\ref{intertwining} satisfies the conditions~a,b,c)
of~Lemma~\ref{lem-byun}. The conditions~a,b) are obvious, so we just need to prove the condition~c).
For that it is enough to construct a collection of objects of $D(B\backslash G/B)^{\BZ/2}$ such that:

\medskip
($i$) These objects go to images of projective objects in $(\HC(\fg^\svee)^{\widehat\lambda})^{\BZ/2}$ under $\Psi_{\lambda}$.

($ii$) These objects are direct sums of images of $\IC_w$'s and the image of every $\IC_w$ occurs as such a direct summand.

\medskip
\noindent
Recall that $D(B\backslash G/B)^{\BZ/2}$ is obtained from $(D(G(\CK))^{(G(\CO)\times G(\CO))\rtimes \BC^{\times}})^{\BZ/2}$
by applying the following procedures:

1) Completion at $a=(\lambda,-w_0\lambda,1)$.

2) Taking fiber over 1 with respect to the last coordinate.

3) Shift by $(\lambda,-w_0\lambda)$.

\noindent
Thus we have a natural functor $(D(G(\CK))^{(G(\CO)\times G(\CO))\rtimes \BC^{\times}})^{\BZ/2}\to D(B\backslash G/B)^{\BZ/2}$; this functor is given
by restriction to $t^{\lambda} G t^{-\lambda}$ and then applying functors~1),~2),~3) above (note that we know that $*$-restriction leads to the
same result as $!$-restriction after applying 1). Let us apply this functor to the images of perverse sheaves in
$(D(G(\CK))^{G(\CO)\times G(\CO)})^{\BZ/2}$. By the geometric Satake equivalence the category of such perverse sheaves is equivalent to
$\Rep(G^{\vee})$. If $V$ is a representation of $G^{\vee}$ then by definition if we apply $\Psi_{\lambda}$ to the image of $V$ in
$D(B\backslash G/B)^{\BZ/2}$ we get the following object of $(\HC(\fg^\svee)^{\widehat\lambda})^{\BZ/2}$: first we consider
$U(\fg^\svee)\otimes V$ with its natural Harish-Chandra bimodule structure and then we complete at $(\lambda,\lambda)$.
Such a bimodule is easily seen to be projective. This shows $(i)$.

Let us now show $(ii)$. By Corollary \ref{braden} we can use the hyperbolic restriction $\Phi$ to $t^{\lambda} G t^{-\lambda}$ as our functor
between the corresponding completions of $(D(G(\CK))^{G(\CO)\times G(\CO)})^{\BZ/2}$ and $D(B\backslash G/B)^{\BZ/2}$. But according to \cite{brad}
this functor maps semi-simple complexes to semi-semisimple ones. This implies the first part of~$(ii)$, and the 2nd part is obvious.


\subsection{Independence of $\lambda$}
\label{indep}
We conclude the paper by showing that the equivalence of Theorem \ref{main} is independent of the choice of the coweight $\lambda$. We give it here since the argument is similar to what is done in this Section.

For this it is enough to do the following. Let $\lambda$ and $\mu$ be two regular coweights of $G$ and let $\nu=\mu-\lambda$. It follows from the discussion at the end of~\S\ref{idea} that the completions of the category $((Q_{\hbar=1}(\fX^{\vee})\mod)^{G^{\vee}})^{\BZ/2}$ at $\lambda$ and at $\mu$ are canonically equivalent (up to the shift of the $\Sym(\grt)$-equivariant structure by $\nu$). We claim that this equivalence can be explicitly seen in the following way.  Let $V$ denote the irreducible finite-dimensional representaion of $G^{\vee}$ for which $\nu$ is an extremal weight and let $\HC(V)=U(\grg^{\vee})\otimes V$ be the corresponding Harish-Chandra bimodule. We have the functor $F_V$ from $((Q_{\hbar=1}(\fX^{\vee})\mod)^{G^{\vee}})^{\BZ/2}$ to itself given by tensor product with $\HC(V)$ over $U(\grg^{\vee})$ (which is the same as just tensor product with $V$ over $\BC$. It is easy to see that $F_V$ induces an equivalence of categories after taking completions at $\lambda$ and $\mu$ respectively, and moreover the resulting equivalence between the two completions at $\lambda$ and at $\mu$ is the one discussed above.

On the other hand, let us go to the LHS of the equivalence of Theorem \ref{main}. Let $\calS_V$ denote the object
of the category $D(G(\CK))^{(G(\CO)\times G(\CO))\rtimes \BC^{\times}}$ corresponding to $V$ under the geometric Satake
equivalence. Note that the functor $F_V$ corresponds to convolution with $\calS_V$ (followed by specialization of
$\hbar$ to 1). Let $i_{\lambda}\colon \fY\to\fY(\CK)$ be the embedding sending every $y$ to $t^{\lambda}y$.
Thus it follows from the above discussion that it is enough to show that the diagram
$$
\begin{CD}
D(\fX(\CK))^{G(\CO)\rtimes \BC^{\times}} @> \star \calS_V >> D(\fX(\CK))^{G(\CO)\rtimes \BC^{\times}} \\
@Vi_{\lambda}^* VV    @VV i_{\mu}^* V \\
D(\fY)^B @>\text{shift by $\nu$} >> D(\fY)^B
\end{CD}
$$

becomes commutative after

a) completion of the left column at $(\lambda,1)$,

b) completion of the right column at $(\mu,1)$,

c) replacing $\calS_V$ by its image in the completion of $D(G(\CK))^{(G(\CO)\times G(\CO))\rtimes \BC^{\times}}$ at $(-\lambda,\mu,1)$.

\noindent
However, we know that after the above completions the vertical arrows become equivalences and the upper horizontal arrow becomes just tensor product with the restriction of $\calS_V$ to the point $t^{\nu}$ which is one dimensional since $\nu$ was assumed to be an extremal weight of the irreducible representation $V$.

\subsection{Isomorphism of $W$-modules}\label{Lagrange}
Assume first that $\fX=T^*\fY$ where $\fY$ is a $G$-variety satisfying the condition that the stabilizer in $B$ of a general point is
connected. We have the natural map from the Grothendieck group of the category of {\em locally compact} objects in
$D(\fY)^B=D(\fY\times\calB)^G$ to $H(\widetilde{\Lambda})$ given by the characteristic cycle, and it follows from~Lemma~\ref{tim}
that this map is an isomorphism. The Grothendieck group of the category of {\em compact} objects in $D(\fY)^B=D(\fY\times\calB)^G$
is isomorphic to the dual space $H(\widetilde{\Lambda})^*$. The duality is given by the Euler characteristic of $R\Hom_{D(\fY)^B}(?,?)$.
Furthermore, the Grothendieck group of {\em compact} objects of $D(\fY)^{B_\mon}$
is isomorphic to $H(\widetilde{\Lambda})$. Hence if both $\fX$ and $\fX^{\vee}$ are of cotangent
type,~Conjecture~\ref{lagrangian} follows by~Corollary~\ref{hecke-mod}. This proves it for instance in the
examples of~\S\S\ref{js},\ref{mir} as well as in the example of~\S\ref{glmn} for $M=N-1$.

We can do something a bit more general.
Assume as before that $\fX=T^*_\psi \fY$ such that the $B$-stabilizer of a general point in $\fY$ is connected. Then we still
get an isomorphism between the Grothendieck group of the category of $\psi$-twisted $B$-equivariant $D$-modules on $\fY$ (which are the
same as $B\times (\GG_a,\chi))$-equivariant $D$-modules on $\psi$ where $\chi$ denotes the identity homomorphism $\GG_a\to \GG_a$)  --
we just need to say that the group generated by conormal bundles to the relevant $G\times \GG_a$-orbits in $T^*\psi\times T^*\calB$ maps
isomorphically onto the group generated by components of $\widetilde{\Lambda}$ under the map which sends every cycle to its intersection
with the preimage of 1 under the moment map with respect to $\GG_a$ followed by projection to $\fX\times T^*\calB$.
This proves~Conjecture~\ref{lagrangian} in the case of~\S\ref{glmn}.

In fact, we can relax the above requirements a little further (so called {\em pseudospherical} case).
Namely, it is enough to require that $\fX$ is as above, and $\fX^{\vee}=T^*_{\psi^{\vee}}\fY^{\vee}$ where $\fY^{\vee}$ is
just a $B^{\vee}$-variety and $\psi^{\vee}$ is a $B^{\vee}$-equivariant torsor on $\fY^{\vee}$, provided that

a) The quantization $Q_{\hbar=1}(\fX^\vee)$ is given by the quantum hamiltonian reduction of the algebra of differential operators on $\psi^{\vee}$ with respect to $\GG_a$.

b) The assertion of Lemma \ref{tim} holds in this case (note that we can not formally apply~Lemma~\ref{tim} here,
since $\fY^\vee$ in this case is not a $G^{\vee}$-variety).

These requirements are satisfied in the case of \S\ref{bessel}. Namely, recall that $G^\vee=\SO(2n+2)\times\Sp(2n)$
or $\SO(2n)\times\Sp(2n)$. We consider the second case, the first one being absolutely similar.
We choose a Lagrangian subspace $L\subset\BC^{2n}_-$ invariant with respect to the Borel
subgroup $B^\vee\subset G^\vee$. We set $L^*=\BC^{2n}_-/L$, and
$\fY^\vee=\BC^{2n}_+\otimes L^*$. Finally, $\psi^\vee$ is the descent of the Heisenberg extension
$0\to\BG_a\to\on{Heis}\to\BC^{2n}_+\otimes\BC^{2n}_-\to0$ to $\fY^\vee$.

We have a natural homomorphism (characteristic cycle) from
$K((D(\fY^\vee)^{B^\vee,\psi^\vee})^{\BZ/2}_{\on{loc.comp}})$ to $H(\Lambda^\vee)$ (from the Grothendieck group of the category of locally
compact objects to the top Borel-Moore homology).
According to~\cite[Remark~4.3.2]{fgt}, a basis in the latter group is
formed by the fundamental cycles of twisted conormal bundles of relevant $B^\vee$-orbits in $\fY^\vee$. It follows that the
above homomorphism of $K$-groups is surjective. However,
\begin{multline*}\dim H(\Lambda^\vee)=\sharp\on{Irr}\Lambda^\vee=\sharp\on{Irr}\Lambda=\dim H(\Lambda)
  =\dim K((D(\fY)^{B_{\on{mon}},\psi})^{\BZ/2}_{\on{comp}})\\ =\dim K((D(\fY^\vee)^{B^\vee,\psi^\vee})^{\BZ/2}_{\on{comp}})=
\dim K((D(\fY^\vee)^{B^\vee,\psi^\vee})^{\BZ/2}_{\on{loc.comp}}).\end{multline*}
Here the second equality is~\cite[Proposition~5.1.1]{fgt}, the fourth equality follows from~Lemma~\ref{tim}, the
fifth equality follows from~Theorem~\ref{main} combined with~\cite[Theorem~2.2.1]{bft1}, and the last equality follows from the
duality discussed in the beginning of this section.
We conclude that the surjective homomorphism $K((D(\fY^\vee)^{B^\vee,\psi^\vee})^{\BZ/2}_{\on{loc.comp}})\twoheadrightarrow H(\Lambda^\vee)$
must be an isomorphism. (In particular, the stabilizers in $B^\vee$ of points in relevant orbits in $\fY^\vee$ are all connected.)
All in all, we obtain an isomorphism of $W$-modules in~Conjecture~\ref{lagrangian}.

\end{document}